\documentclass[a4paper,12pt]{article}
\usepackage{amsmath,amsthm,amssymb,latexsym,epsfig,graphicx}
\title{{\bf A general form of Green Formula and Cauchy Integral Theorem }}
\author{\Large{\Large Juli\`{a} Cuf\'{\i} and Joan Verdera}}
\newtheorem*{teor}{Theorem}
\newtheorem*{teor2}{Theorem (pointwise version)}
\newtheorem*{teor3}{Theorem (Carmona and Cuf\'{\i})}
\newtheorem*{cor}{Corollary}

\newtheorem{co}{Corollary}
\newtheorem{lemma}[co]{Lemma}
\newtheorem*{ML}{Main Lemma}

\theoremstyle{definition}

\newtheorem*{gracies}{Acknowledgements}

\newcommand{\T}{\mathbb{T}}
\newcommand{\C}{\mathbb{C}}

\newcommand{\In}{\operatorname{Ind}}
\newcommand{\Ing}{\operatorname{Ind}(\gamma,z)}

\newcommand{\dbar}{\overline{\partial}}

\begin{document}

\date{}

\maketitle

\begin{abstract}
We prove a general form of Green Formula and Cauchy Integral Theorem
for arbitrary closed rectifiable curves in the plane. We use
Vitushkin's localization of singularities method and a decomposition
of a rectifiable curve in terms of a sequence of Jordan rectifiable
sub-curves due to Carmona and Cuf\'{i}.
\end{abstract}

\section{Introduction}

In this paper we prove a general form of Green Formula and Cauchy
Integral Theorem for arbitrary closed rectifiable curves in the
plane. A closed rectifiable curve is a complex valued mapping
$\gamma$ of bounded variation defined on the unit circle $\T.$ We
adopt the standard abuse of notation consisting in denoting by
$\gamma$ also the image of the unit circle under the mapping. Recall
that the winding number or index of a closed rectifiable curve
$\gamma$ with respect to a point $z \notin \gamma$ is
\begin{equation*}\label{index}
\operatorname{Ind}(\gamma,z)= \frac{1}{2 \pi i} \int_\gamma
\frac{dw}{w-z}.
\end{equation*}

Set
\begin{equation*}\label{D}
D = \{z \in \C: \operatorname{Ind}(\gamma,z)\neq 0 \}
\end{equation*}
and
\begin{equation*}\label{D_0}
D_0 = \{z \in \C: \operatorname{Ind}(\gamma,z)= 0 \}.
\end{equation*}
The sets $D$ and $D_0$ are open and a countable union of connected
components of $\C \setminus \gamma.$
 We let $dA$ stand for planar Lebesgue measure and $\dbar =
\partial/\partial\overline{z}$ for the usual Cauchy-Riemann
operator. We then have the following.

\begin{teor}\label{teorema}
Let $\gamma$ be a closed rectifiable curve and let $f$ be a
continuous function on $D \cup \gamma$ such that the $\dbar$
derivative of $f$ in $D$, in the sense of distributions, belongs to
$L^2(D).$ Then
\begin{equation*}\label{identitat}
\int_\gamma f(z) \,dz = 2 i \int_{D} \dbar f(z)\In(\gamma,z)\,d
A(z).
\end{equation*}
\end{teor}

If $D$ is empty then the right hand side in \eqref{identitat} is $0$
and the identity is straightforward. Notice that the integral in the
right hand side is absolutely convergent because the function $\Ing$
is in $L^2(\C).$ This was proven with the best constant in
\cite{CC1}. It is also a consequence of the Sobolev imbedding
Theorem for $p=1$ and the fact that $\Ing$ is a function of bounded
variation. Indeed
\begin{equation*}\label{variaciofitada}
\dbar (\Ing) = \frac{dz}{2 i} \hspace{0.6cm}\text{and}
\hspace{0.6cm}
\partial(\Ing)= - \frac{\overline{dz}}{2 i}.
\end{equation*}
 It is not true in general that $\Ing \in L^p(\C)$ for some $p > 2.$
 Our proof works also under the assumption that $\dbar f(z)\In(\gamma,z) \in L^1(D).$
\begin{cor}\label{cor}
Let $\gamma$ be a closed rectifiable curve and let $f$ be a
holomorphic function on $D$ which is continuous on $D \cup \gamma.$
Then
$$
\int_\gamma f(z) \,dz = 0.
$$
\end{cor}

With the extra hypothesis that $\Ing$ is bounded on $\C \setminus
\gamma$  the Corollary was proven by N\"{o}beling in 1949 \cite{N}. In
fact, the corollary is proven  in \cite{M} as a consequence of an
approximation theorem of $\gamma$ by chains formed by boundaries of
squares contained in $D.$ Michael's approximation theorem coupled
with a regularization argument can be used to give a proof of the
Theorem above. Our proof, which we found before becoming aware of
\cite{M}, keeps the curve fixed and, instead, the function $f$ is
suitably approximated.

Combining the Theorem with a well known result of Fesq \cite{F} (see
also \cite{Co}) one obtains an appealing statement in which no
distributions theory is involved. Fesq proved the following. Assume
that a function $f $ is defined and continuous on an open set
$\Omega$ and has partial derivatives $\partial f/\partial x$ and
$\partial f/\partial y$ at each point of $\Omega \setminus E$, where
$E$ is a countable union of closed sets of finite length (one
dimensional Hausdorff measure). Assume further that the measurable
function $\dbar f (z)= \frac{1}{2}(\frac{\partial f}{\partial x}(z)+
i \frac{\partial f}{\partial y}(z))$ defined  for $z \in \Omega
\setminus E$, belongs to $L_{loc}^1(\Omega)$. We emphasize that now
$\dbar f(z)$ is not defined in the sense of distributions but only
pointwise. Then
\begin{equation}\label{Fesq}
\int_{\partial Q} f(z)\,dz = 2 i \int_Q \dbar f(z)\,dA(z),
\end{equation}
for each square $Q$ with closure contained in $\Omega.$ It is not
difficult to realize that \eqref{Fesq} implies that the pointwise
$\dbar$-derivative of $f$ on $\Omega$ is indeed the distributional
$\dbar$-derivative of $f$ on $\Omega.$ For the sake of completeness
a proof of this simple fact is presented in section 4. Therefore we
obtain the following variation of the Theorem.

\begin{teor2}\label{teorema2}
Let $\gamma$ be a closed rectifiable curve. Let $f$ be a continuous
function on $D \cup \gamma$ whose partial derivatives $\partial
f/\partial  x$ and $\partial f/\partial y$ exist at each point of $D
\setminus E$, where $E$ is a countable union of closed sets of
finite length (one dimensional Hausdorff measure), and such that
$\dbar f \in L^2(D),$ where $\dbar f$ is defined pointwise almost
everywhere on $D$. Then
$$
\int_\gamma f(z) \,dz = 2 i \int_{D} \dbar f(z)\In(\gamma,z)\,d
A(z).
$$
\end{teor2}

In \cite{M} a weaker version of the preceding result is proven under
the extra assumption that $\partial f (z)/\partial x \In(\gamma,z)$
and $\partial f (z)/\partial y \In(\gamma,z)$ are in $L^1(D).$

In section 2 we present the proof of the Theorem and we leave for
section 3 the discussion of the Main Lemma. The main tool in the
proof is the method of separation of singularities due to Vitushkin
(see \cite{G}, \cite{V} or \cite{Vi}). This is quite natural because
in many instances Cauchy Integral Theorem is reduced to the case in
which more regular functions are involved via uniform approximation
of the given data. For instance, if $D$ is a Jordan domain and
$\gamma$ its boundary, then one can approximate $f$, uniformly on
$D$, by polynomials in $z$, for which the result is obvious.

Vitushkin's method produces a large finite sum. The terms in this
sum are divided into three classes and in estimating the class which
involves more directly the curve we decompose $\gamma$ in a sum, in
most cases infinite, of Jordan curves. This decomposition is a
consequence of \cite [Theorem 4] {CC2} and reads as follows.

\begin{teor3}\label{CC}
For each closed rectifiable curve $\gamma$ such that $D \neq
\emptyset$ there exists a sequence (maybe finite) of Jordan curves
$(\gamma_n)_{n=1}^\infty$ with the property that $\gamma_n \subset
\gamma,$ $dz_{\gamma} = \sum_{n=1}^\infty dz_{\gamma_n}$ and
$\sum_{n=1}^\infty l(\gamma_n) \leq l(\gamma).$
\end{teor3}

Here $l(\gamma)$ stands for the length of the curve $\gamma$ and
$dz_{\gamma} = \sum_{n=1}^\infty dz_{\gamma_n}$  means that
$\int_\gamma f(z)\,dz = \sum_{n=1}^\infty \int_{\gamma_n} f(z)\,dz
,$ for any continuous function $f$ on $\gamma.$

The estimate we are looking for is then reduced to the case of a
Jordan curve, which is dealt with in the Main Lemma.

A word on the existing literature on Green Formula and Cauchy
Integral Theorem is in order. Burckel, in his well-known
comprehensive book on Classical Complex Analysis \cite[p. 341]{B},
states that the more general Cauchy Theorem he knows is that due to
N\"{o}beling (\cite{N}), in which the index of the curve is assumed to
be bounded. He seems to be unaware of Michael's article \cite{M},
which is apparently widely unknown. We believe that the general form
of Green Formula and Cauchy Integral Theorem involving arbitrary
rectifiable curves and functions defined in the minimal domain $D
\cup \gamma$ deserves to be better known.

\section{Proof of the Theorem}
We first describe Vitushkin's scheme to separate singularities of
functions. The first step is the construction of partitions of unity
subordinated to special coverings of the plane by discs of equal
radii.

\begin{lemma}\label{particio}
Given any $\delta > 0$ there exists a countable family of discs
$(\Delta_j)$ of radius $\delta$ and a family of functions $
\varphi_j \in C_0^\infty(\Delta_j)$ such that

(i) $\C= \cup_j  \Delta_j$.

(ii) The family $(\Delta_j)$ is almost disjoint, that is, for some
constant $C$ each $z \in \C$ belongs to at most $C$ discs $\Delta_j$
(in fact we can take $C=21$).

(iii) $\sum_{j} \varphi_j = 1, \, 0 \le \varphi_j$ and $|\nabla
\varphi_j(z)| \le C \delta^{-1}, \;\; z\in \C, $  where $C$ is an
absolute constant.
\end{lemma}

For the proof, take a grid of squares of side length $\delta/2$ and
regularize the characteristic function of each square with an
appropriate approximation of the identity (see \cite[p.440-441]{V},
\cite{G} or \cite{Vi}.

Given a compactly supported continuous function $f$ on the plane,
set
\begin{equation}\label{efejota}
f_j = \frac{1}{\pi z}* \varphi_j \,\dbar f,
\end{equation}
which makes sense, because it is the convolution of the compactly
supported distribution $\varphi_j \,\dbar f$ with the locally
integrable function $\frac{1}{\pi z}.$ Since $\frac{1}{\pi z}$ is
the fundamental solution of the differential operator $\dbar$, we
have $\dbar f_j = \varphi_j \,\dbar f$ and thus $f_j$ is holomorphic
where $f$ is and off a compact subset of $\Delta_j.$ It is easy to
see that
$$
f_j(z)=\frac{1}{\pi} \int \frac{f(w)-f(z)}{w-z}\, \dbar \varphi_j(w)
\, dA(w)
$$
and hence
$$
|f_j(z)| \le C\, \omega(f,\delta), \quad z \in \C,
$$
where $C$ is an absolute constant and $\omega(f,\delta)$ is the
modulus of continuity of $f.$ Since the family of functions
$(\varphi_j)$ is a partition of the unity,
\begin{equation}\label{separation}
f= \sum_{j} f_j.
\end{equation}

If one defines a singularity of $f$ as a point in the support of
$\dbar f,$ then clearly the effect of \eqref{separation} is to
distribute the singularities of $f$ among the discs $\Delta_j.$
Notice that the sum in \eqref{separation} contains only finitely
many non-zero terms, because the support of $\dbar f$ is compact.

Having set up these preliminaries, let us start the proof of the
Theorem. We will consider only the case in which $D$ is not empty;
otherwise the conclusion is straightforward. Extend the function $f$
to a compactly supported continuous function on the plane, fix a
$\delta
> 0$ and apply \eqref{separation}. We divide the indexes $j$ into
the following three classes :
\begin{align*}
I &= \{j : \Delta_j \subset D \},\\*[3pt]
II &= \{j : \Delta_j \cap \gamma \neq \emptyset \},
\intertext{and}
III &= \{j : \Delta_j \subset D_0 \}.
\end{align*}

For $j \in III$ we have $\int_\gamma f_j(z)\,dz =0,$ because $f_j$
is holomorphic on an open set in which $\gamma$ is homologous to
zero. Hence
\begin{equation*}\label{sumaIII}
\sum_{j \in III} \int_\gamma f_j(z)\,dz =0.
\end{equation*}

For $j \in I$ use the definition of $f_j$ in \eqref{efejota} and
Fubini's theorem to get
\begin{equation*}\label{grupI}
\int_\gamma f_j(z)\,dz = 2 i \int_{\C} \varphi_j(w)\, \dbar
f(w)\In(\gamma,w)\,d A(w).
\end{equation*}

Adding up in $j \in I$ one obtains
\begin{equation*}\label{sumaI}
\sum_{j \in I} \int_\gamma f_j(z)\,dz = 2 i \int_{\C} \left(\sum_{j
\in I} \varphi_j(w)\right)\, \dbar f(w)\In(\gamma,w)\,d A(w).
\end{equation*}
Since
$$
\lim_{\delta \rightarrow 0} \sum_{j \in I} \varphi_j(w) = 1, \quad w
\in D
$$
and $\dbar f(w)\In(\gamma,w) \in L^1(D),$ it follows, by dominated
convergence, that
\begin{equation*}\label{limsumaI}
\lim_{\delta \rightarrow 0}\sum_{j \in I} \int_\gamma f_j(z)\,dz = 2
i \int_{\C} \dbar f(w)\In(\gamma,w)\,d A(w).
\end{equation*}

Therefore
$$
\int_\gamma f(z) \,dz = 2 i \int_{\C} \dbar f(z)\In(\gamma,z)\,d
A(z) + \lim_{\delta \rightarrow 0}\sum_{j \in II} \int_\gamma
f_j(z)\,dz
$$
and so to complete the proof it is enough to check that the limit in
the above right hand side vanishes. This follows from the inequality
\begin{equation}\label{sumaII}
\sum_{j \in II} \left|\int_\gamma f_j(z)\,dz \right| \le
C\,\omega(f,\delta)\, l(\gamma) + \eta(\delta),
\end{equation}
where $C$ is an absolute constant and $\eta(\delta)$ a function
which tends to zero with $\delta$. To show \eqref{sumaII} fix $j \in
II.$ One has
\begin{equation}\label{intj}
\begin{split}
\left|\int_\gamma f_j(z)\,dz \right| &= \left|\int_{\gamma \cap
\overline{\Delta}_j} f_j(z)\,dz \right|+ \left|\int_{\gamma \cap
(\overline{\Delta}_j)^c} f_j(z)\,dz \right|
\\*[5pt] & \le C\,\omega(f,\delta)\,l(\gamma \cap
\overline{\Delta}_j) + \left|\int_{\gamma \cap
(\overline{\Delta}_j)^c} f_j(z)\,dz \right|.
\end{split}
\end{equation}
Observe that adding up on $j \in II$ the first terms in right hand
side of the inequality above one gets the desired estimate, namely,
$$
\sum_{j \in II} \omega(f,\delta)\,l(\gamma \cap \overline{\Delta}_j)
\le C \, \omega(f,\delta)\, l(\gamma),
$$
where we used that the family of discs $\Delta_j$ is almost
disjoint. However, the obvious estimate for the second term
$$
\left|\int_{\gamma \cap (\overline{\Delta}_j)^c} f_j(z)\,dz \right|
\le C \,\omega(f,\delta) \, l(\gamma \cap (\overline{\Delta}_j)^c)
$$
does not lead anywhere because the length of $\gamma$ off the disc
$\overline{\Delta}_j$ is not under control. To overcome this
difficulty we resort to the next lemma.

\begin{ML}\label{ML}
Let $\Gamma$ be a closed rectifiable Jordan curve, $\Delta$ a disc
of radius $\delta$ and $h$ a bounded continuous function on $\C,$
holomorphic off a compact subset of $\Delta.$ Then
\begin{equation}\label{foradisc}
\left|\int_{\Gamma \cap (\overline{\Delta})^c}  h(z)\,dz \right| \le
2 \pi \, \|h\|_\infty \,\delta,
\end{equation}
where $\|h\|_\infty $ is the supremum norm of $h$ on the whole
plane.
\end{ML}

We postpone the proof of the Main Lemma to next section and we
proceed to finish the proof of the Theorem.

First of all we prove \eqref{sumaII} with $\eta \equiv 0$ assuming
that $\gamma= \Gamma$ is a Jordan curve.

Consider the case $\delta \le \frac{1}{2}
\operatorname{diam}(\Gamma).$ Since there is a point in $\Gamma \cap
\partial \Delta_j$ the length of the curve in the disc $3 \Delta_j$
is larger than $2 \delta. $ Combining \eqref{intj} with the Main
Lemma applied to $f_j$ and $\Delta_j$ we conclude that
$$
\left|\int_\Gamma f_j(z)\,dz \right| \le C \,\omega(f,\delta)\,
l(\Gamma \cap (3 \Delta_j),
$$
which yields \eqref{sumaII} with $\eta \equiv 0,$ because the family
of discs $(3 \Delta_j)$ is almost disjoint.

If $\delta
> \frac{1}{2} \operatorname{diam}(\Gamma),$ then the number of discs
$\Delta_j$ that intersect $\Gamma$ is less than an absolute
constant. Thus \eqref{sumaII} with $\eta \equiv 0$ and $\gamma$
replaced by $\Gamma$ also holds in this case.

 Now we will reduce the proof of
\eqref{sumaII} to the case of a Jordan curve by appealing to the
decomposition theorem of \cite{CC2} mentioned in the introduction.
There is a sequence (maybe finite) of rectifiable Jordan curves
$(\gamma_n)_{n=1}^\infty$ such that $\gamma_n \subset \gamma,$
$dz_{\gamma} = \sum_{n=1}^\infty dz_{\gamma_n}$ and
$\sum_{n=1}^\infty l(\gamma_n) \leq l(\gamma).$  We need to
decompose $II$ into two subsets. Let $II_0$ be the set of indices $j
\in II$ such that $D_j$ does not intersect any $\gamma_n$ and $II_1
= II \setminus II_0. $  Since the set of indices $II$ is finite, we
have
\begin{equation*}\label{sumaII2}
\begin{split}
\sum_{j \in II_1} \left|\int_\gamma f_j(z)\,dz \right|& \le \sum_{j
\in II_1} \sum_{n=1}^\infty \left|\int_{\gamma_n} f_j(z)\,dz \right|
\\*[5pt] & =  \sum_{n=1}^\infty
\sum_{j \in II_1} \left|\int_{\gamma_n} f_j(z)\,dz \right|.
\end{split}
\end{equation*}
Given $n=1,2,...$ set $II_n = \{j : \Delta_j \cap \gamma_n \neq
\emptyset \}.$  Each $j \in II_1$ belongs to at least one $II_n$,
but may belong to several. Taking into account this remark in the
first inequality below and applying \eqref{sumaII} with $\eta\equiv
0$ to the Jordan curve $\gamma_n$ and the function $f$ in the second
one, we get
\begin{equation*}\label{sumaIIfinal}
\begin{split}
\sum_{n=1}^\infty \sum_{j \in II_1} \left|\int_{\gamma_n} f_j(z)\,dz
\right| & \le \sum_{n=1}^\infty \sum_{j \in II_n}
\left|\int_{\gamma_n} f_j(z)\,dz \right|
\\*[5pt] & \le  \sum_{n=1}^\infty C \,\omega(f,\delta)\,l(\gamma_n)
\\*[5pt] & \le  C \,\omega(f,\delta)\,l(\gamma),
\end{split}
\end{equation*}
which is the right estimate.

We turn now our attention to the sum over $II_0.$  If $j \in II_0$
then $\In(\gamma,z)= \sum_{n=1}^\infty \In(\gamma_n,z)$ is constant
on $\Delta_j \setminus \gamma.$ Hence $\Delta_j \setminus \gamma
\subset D_0$ or $\Delta_j \setminus \gamma \subset D.$

In the first case we argue as we did for $j \in III.$ The infinite
cycle $\sum_{n=1}^\infty \gamma_n$ is homologous to $0$ in an open
set on which $f_j$ is holomorphic and therefore $\int_\gamma
f_j(z)\,dz = \sum_{n=1}^\infty \int_{\gamma_n} f_j(z)\,dz =0.$ This
follows by the usual argument to prove Cauchy's Theorem for finite
cycles.

To settle the case $\Delta_j \setminus \gamma \subset D$ we need to
know that the $\dbar$-derivative of $f$ on $\Delta_j$ does not
charge the set $\gamma \cap \Delta_j.$

\begin{lemma}\label{lemagamma}
Let $\Delta$ be an open disc and assume that $\Delta \setminus
\gamma \subset D.$ Then the $\dbar$-derivative of $f$ on $\Delta$ in
the distributions sense is the function \, $\dbar f(z)
\,\chi_{\Delta \setminus \gamma}(z).$
\end{lemma}
\begin{proof}
It is shown in Appendix 1(section 4) that the conclusion of the
Lemma follows from the identity
\begin{equation}\label{Q}
\int_{\partial Q} f(z)\,dz = 2 i \int_Q \dbar f(z)\, \chi_{\Delta
\setminus \gamma}(z) \,dA(z),
\end{equation}
for each closed square $Q$ with sides parallel to the coordinates
axis contained in $\Delta .$ To prove \eqref{Q} subdivide $Q$ in
dyadic sub-squares. At the $n$-th generation one has $4^n$ dyadic
sub-squares of $Q$, denoted by $Q_k, \; 1 \le k \le 4^n.$  Their
side length is $L \,4^{-n},$ $L$ being the side length of $Q.$ Set
$I=\{k : Q_k \cap \gamma = \emptyset \}$ and $J=\{k : Q_k \cap
\gamma \neq \emptyset \}.$ Then
\begin{equation*}\label{integraldQ}
\begin{split}
\int_{\partial Q} f(z)\,dz & = \sum_{k \in I} \int_{\partial Q_k}
f(z)\,dz + \sum_{k \in J} \int_{\partial Q_k} f(z)\,dz
\\*[5pt] & =  \sum_{k \in I} 2 i \int_{Q_k} \dbar
f(z)\,dA(z) + \sum_{k \in J} \int_{\partial Q_k} \left(f(z)-f(z_k)
\right) \,dz
\\*[5pt] & \equiv T_1(n)+T_2(n),
\end{split}
\end{equation*}
where $z_k$ is any point in $Q_k$ and the last identity is a
definition of $T_1(n)$ and $T_2(n).$ On the one hand it is clear
that
$$
\lim_{n \rightarrow \infty} T_1(n) = 2i \int_Q \dbar f(z)\,
\chi_{\Delta \setminus \gamma}(z) \,dA(z)
$$
and on the other hand, setting $\epsilon_n = \sqrt{2} L 4^{-n},$ we
have
$$
|T_2(n)| \le \omega(f,\epsilon_n) \sum_{k \in J} l(\partial Q_k).
$$
Take $n$ big enough so that $\operatorname{diam}(2 Q_k) = \sqrt{2}\,
2 L 4^{-n} < \operatorname{diam}(\gamma).$ Then $ l(\partial Q_k)
\le 8 \, l(2 Q_k \cap \gamma)$ and thus
$$
|T_2(n)| \le C\, \omega(f,\epsilon_n) l(\gamma),
$$
because the family of squares $2 Q_k, 1 \le k \le 4^n,$ is almost
disjoint (with an absolute constant). Letting $n \rightarrow \infty
$ we get \eqref{Q}.
\end{proof}

Denote by $II_2$ the set of indices $j \in II_0$ such that $\Delta_j
\setminus \gamma \subset D.$ For $j \in II_2,$ by Fubini's theorem
and Lemma 2,
\begin{equation*}\label{integralII2}
\begin{split}
\int_\gamma f_j(z)\,dz & = 2 i \int_{\C} \varphi_j(w)\, \dbar
f(w)\,\chi_{\Delta_j \setminus \gamma}(w)\,\In(\gamma,w)\,d A(w)
\\*[5pt] & = 2 i \int_{D} \varphi_j(w)\, \dbar f(w)\In(\gamma,w)\,d
A(w),
\end{split}
\end{equation*}
which is the same relation we found for indices $j\in I.$ It is also
clear that for some Borel subset $E$ of $\gamma,$ $\lim_{\delta
\rightarrow 0}\sum_{j \in II_2} \varphi_j (z) = \chi_E(z), \,z \in
\C.$ Therefore
$$
\sum_{j \in II_0} \left|\int_\gamma f_j(z)\,dz \right| \le 2
\int_{D} \left( \sum_{j \in II_2} \varphi_j(z)\right)\, |\dbar
f(z)\In(\gamma,z)|\,d A(z)\equiv \eta(\delta)
$$
and, since $E \subset \gamma$, $\eta(\delta)$ tends to zero with
$\delta.$ This completes the proof of \eqref{sumaII} and of the
Theorem.

\section{Proof of the Main Lemma}
We  can assume , without loss of generality, that $\Gamma$
intersects the circle $\partial \Delta$ in finitely many points.
Indeed, by the Banach Indicatrix Theorem \cite[p. 225]{Na} applied
to the function of bounded variation $\left|\gamma \right|$, there
is a sequence of numbers $\lambda_n > 1$ with limit $1$ such that
$\Gamma$ intersects $\partial (\lambda_n \Delta)$ in finitely many
points. It is readily shown that the inequality \eqref{foradisc}
follows as soon as one knows it for the discs $\lambda_n \Delta.$
Assume then that $\Gamma$ intersects the circle $\partial \Delta$ in
finitely many points.

We claim that there are finitely many subintervals $(I_k)$ of the
circle $\partial \Delta,$ which are mutually disjoint, such that
\begin{equation}\label{intervals}
\int_{\Gamma \cap (\overline{\Delta})^c}  h(z)\,dz = \sum_k
\epsilon_k \int_{I_k}  h(z)\,dz,
\end{equation}
where $\epsilon_k = \pm 1$ determines the orientation on the
interval $I_k$ ($\epsilon_k =1$ corresponds to the counterclockwise
orientation of $\partial \Delta$). It is plain that
\eqref{intervals} completes the proof of the Main Lemma.

Take a connected component $C$ of $\Gamma \cap
(\overline{\Delta})^c.$ The open Jordan arc $C$ has two end points
on $\partial\Delta$ ,which determine two complementary open
intervals $I_1$ and $I_2$ in $\partial\Delta.$  Each $I_j, j=1,2,$
determines a closed Jordan curve $C_j$ which is the union of $C$ and
the closure of $I_j.$  We claim that one and only one of the two
closed curves $C_j$ has index zero with respect to the center $z$ of
$\Delta.$ This means that the domain enclosed by this Jordan curve
lies completely outside the closed disc $\overline{\Delta}.$ To show
the claim take a path joining $z$ with $\infty$ without touching the
closure of $C.$ Consider the point $w$ where this path leaves the
disc $\overline{\Delta}$ for the last time. Modifying the path by
taking first the segment joining $z$ with $w$, we may assume that it
indeed intersects $\partial \Delta$ at only one point $w$, which
belongs to $I_1$ or $I_2.$ If it lies on $I_2$ then the index of
$C_1$ with respect to $z$ is $0$ and the index of $C_2$ with respect
$z$ is $\pm 1.$ The case $w \in I_1$ is symmetric and so the claim
is proven.

Given a connected component $C$ of $\Gamma \cap
(\overline{\Delta})^c,$ we define $I(C)$ to be the interval $I_j$ in
the preceding discussion such that the domain enclosed by $C_j$ lies
in the complement of $\overline{\Delta}.$

\begin{center}
\includegraphics{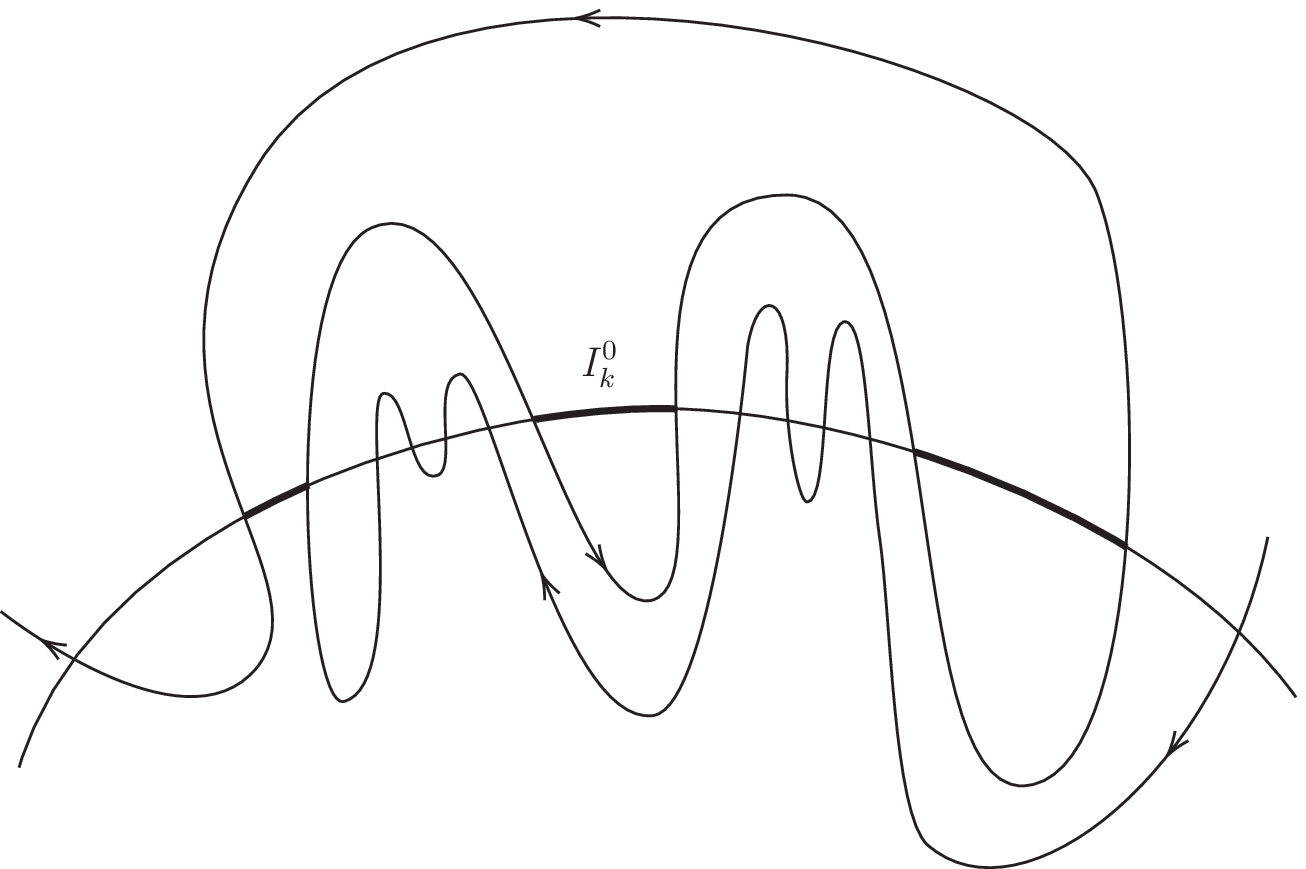}
\end{center}
\begin{center}
Figure 1
\end{center}

If $C_1$ and $C_2$ are two different components of $\Gamma \cap
(\overline{\Delta})^c$ then $I(C_1) \subset I(C_2) $ or $I(C_2)
\subset I(C_1)$ or the intervals $I(C_1)$ and $I(C_2)$ have disjoint
interiors. This ``dyadic" structure allows a classification of the
intervals $I(C)$ in generations. Since there are finitely many
components of $\Gamma \cap (\overline{\Delta})^c$ some of the
intervals $I(C)$ are maximal with respect to inclusion. These are
said to be of generation  $0$ and they form a set $G_0.$ If an
interval $I(C)$ is contained in exactly one interval of $G_0$ it is
said to be of generation $1.$ The intervals $I(C)$ of generation $1$
form a set $G_1.$ In this way we define inductively intervals $I(C)$
of generation $p$ and the corresponding set $G_p$ of intervals of
generation $p$. Clearly there are only finitely many generations
(see Figure 1).

To construct the intervals $(I_k)$ in \eqref{intervals} we proceed
inductively as follows. Notice that the subset of $\partial \Delta$
\begin{equation*}\label{intervals0}
\cup_{I \in G_0} \left( I \setminus \cup_{I\supset J \in G_1} J
\right)
\end{equation*}
is a union of disjoint intervals. Call this intervals $(I_k^0).$ We
then have, for appropriately chosen $\epsilon_k = \pm 1,$
\begin{equation}\label{integral0}
\sum_{I(C) \in G_0} \int_C h(z) \,dz + \sum_{I(C) \in G_1} \int_C
h(z) \,dz = \sum_{k} \epsilon_k \int_{I_k^0} h(z) \,dz.
\end{equation}
This is a consequence of Cauchy's integral theorem applied to the
function $h$ and the Jordan curves defined as follows.  For each
$I(C_0) \in G_0$ define the Jordan curve which consists of the arc
$C_0,$ the arcs $C$ such that $I(C) \in G_1$ and $I(C)\subset
I(C_0)$ and the intervals $I_k^0 \subset I(C_0).$ The subarcs of
$\Gamma$ keep the orientation of $\Gamma$ and the orientation on the
$I_k^0$ can be chosen so that \eqref{integral0} holds because of a
 topological fact that we discuss below.

 We say that $\Gamma$ enters the disc $\Delta$ at the point $\Gamma(t_0)\in
 \partial\Delta$ if there is $\epsilon>0$ with the property that $\Gamma(t) \in
 (\overline{\Delta})^c$ for $t_0-\epsilon < t <t_0$ and $\Gamma(t) \in
 \Delta$ for $t_0< t <t_0 + \epsilon.$  We say that  $\Gamma$ leaves the disc $\Delta$ at the point $\Gamma(t_0)\in
 \partial\Delta$ if there is $\epsilon>0$ with the property that $\Gamma(t) \in
 \Delta$ for $t_0-\epsilon < t <t_0$ and $\Gamma(t) \in
 (\overline{\Delta})^c$ for $t_0< t <t_0 + \epsilon.$  There is a third category of
 points in $\Gamma \cap \partial\Delta,$ namely those with
 the  property that the curve just before and just after the point
 stays either in the disc or in the complement of its closure. We
 will ignore these points. Consider now two points in  $\Gamma \cap
 \partial\Delta$ at which $\Gamma$ enters or leaves the disc and assume
 that in one of the complementary intervals in $\partial\Delta$
 determined by these two points there is no other point at which  $\Gamma$ enters or leaves the
 disc. Then we claim that at one of the two points the curve enters the disc and at the other
 the curve leaves the disc. In other words, it is not possible that either the curve enters the disc at both points or
 that the curve leaves the disc at both points. Before embarking in
 the proof of this claim we remark that, with \eqref{integral0} at
 our disposition, and arguing inductively with the intervals of
 generation 2 and subsequent (if any),  we finally get \eqref{intervals}.

To prove the claim, take two points $A$ and $B$ in $\Gamma \cap
\partial\Delta$ at which $\Gamma$ enters or leaves the disc and
assume
 that in one of the complementary open
 intervals in $\partial\Delta$
 determined by $A$ and $B$ there is no other point at which  $\Gamma$ enters or leaves the
 disc. Then we have to show that the curve enters the disc at one of
 the points $A$ and $B$
 and leaves the disc at the other.

Proceeding by contradiction, we assume that at $A$ and $B$ the curve
leaves the disc (the argument is similar for the case in which the
curve enters the disc at $A$ and $B$ ). Assume also that in the
interval on the circle $\partial \Delta$ which joins $A$ to $B$ in
the clockwise direction there are no points at which the curve
enters or leaves the disc (see Figure 2).

\begin{center}
\includegraphics{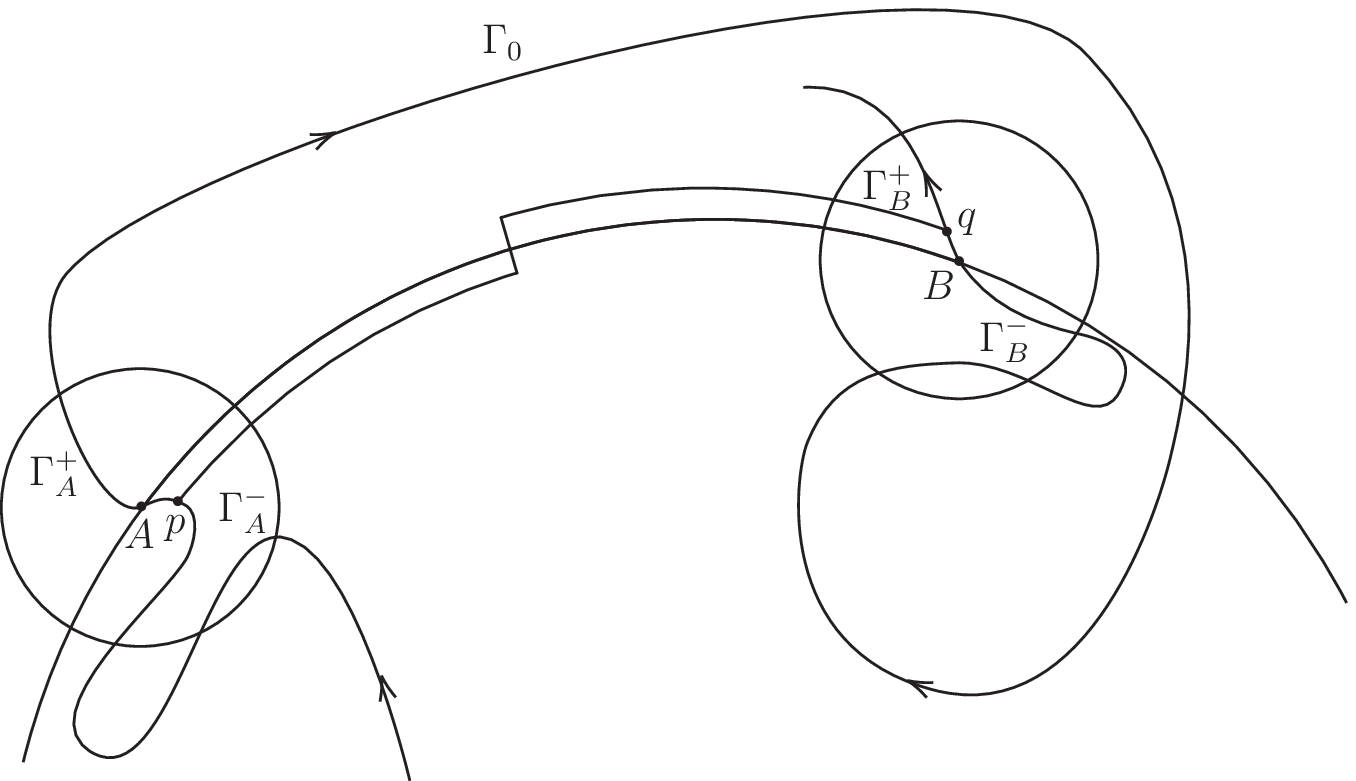}
\end{center}
\begin{center}
Figure 2
\end{center}

\noindent
In view of the definition
that the curve leaves the disc at $A$, there is a disc $D_A$
centered at $A$ and of radius small enough so that $\Gamma \cap D_A$
contains a Jordan arc $\Gamma_A^{-}$ joining $\partial D_A$ to $A$
inside $\Delta$ (that is, $\Gamma_A^{-} \setminus \{A\} \subset
\Delta$) and a Jordan arc $\Gamma_A^{+}$ joining $A$ to $\partial
D_A$ inside $(\overline{\Delta})^c$ (that is, $\Gamma_A^{+}
\setminus \{A\} \subset (\overline{\Delta})^c.$) The same argument
applies to $B$ to produce a disc $D_B$ centered at $B$ and Jordan
arcs $\Gamma_B^{-}$ and $\Gamma_B^{+}$ joining respectively
$\partial D_B$ to $B$ inside $\Delta$ and $B$ to $\partial D_B$
inside $(\overline{\Delta})^c.$ Let $\Gamma_0$ be the Jordan arc
contained in $\Gamma$ joining the end point of $\Gamma_A^{+}$ to the
initial point of $\Gamma_B^{-}.$ Let $\Gamma_{AB}$ be the closed
Jordan curve formed by the union of the $4$ arcs $\Gamma_A^{+},
\Gamma_0 , \Gamma_B^{-}$ and $\widehat{AB}$, where $\widehat{AB}$ is
the interval joining $B$ to $A$ in the circle $\partial\Delta$ with
the counterclockwise orientation. The sub-arcs of $\Gamma$ keep the
orientation provided by the original parametrization of~$\Gamma.$

Out goal is to find points $p \in \Gamma_A^{-}$ and $q \in
\Gamma_B^{+}$ with different index with respect to the Jordan curve
$\Gamma_{AB}.$ This will provide a contradiction, because the
sub-arc of $\Gamma$ starting at $B$ and ending at  $A$ joins $p$ and
$q$ without intersecting $\Gamma_{AB},$ which means that $p$ and $q$
have the same index with respect to $\Gamma_{AB}.$

Let us proceed to the definition of $p$ and $q.$ Since $\Gamma_0$
and $\widehat{AB}$ are disjoint compact sets there is $\delta > 0$
such that $U_\delta \cap \Gamma_0 = \emptyset$, where $U_\delta =
\{z : \operatorname{dist}(z, \widehat{AB}) < \delta \}.$ Take $p \in
U_\delta \cap \Gamma_A^{-}$ and $q \in U_\delta \cap \Gamma_B^{+}.$
The next step is to construct a Jordan arc joining $p$ and $q$,
which intersects $\Gamma_{AB}$ only once, so that $p$ and $q$ have
different index with respect to $\Gamma_{AB}.$ Start at $p$ and
follow the circle concentric with $\partial\Delta$ which contains
$p$ in the clockwise direction until we are under the middle point
of the interval $\widehat{AB}.$ Continue along the ray emanating at
the center of $\Delta$ towards $\partial\Delta,$ cross
$\partial\Delta$ and proceed until you touch the circle concentric
with $\partial\Delta$ containing $q.$ Then follow that circle until
you get to $q.$ Obviously you cross $\Gamma_{AB}$ once through
$\widehat{AB}$, but there is no other intersection with
$\Gamma_{AB}.$ The proof is now complete.

\section{Appendix }
In this section we prove the following.
\begin{lemma}\label{Fesq1}
Assume that $f$ is a continuous function on an open set $\Omega$
such that its partial derivatives exist almost everywhere in
$\Omega$ and $\overline{\partial} f \in L^1_{loc}(\Omega).$ Assume
further that
\begin{equation}\label{Fesq2}
\int_{\partial Q} f(z)\,dz = 2 i \int_Q \dbar f(z)\,dA(z),
\end{equation}
for each closed square $Q \subset \Omega$ with sides parallel to the
coordinates axis. Then the pointwise $\dbar$ derivative of $f$ on
$\Omega$ is indeed the distributional derivative of $f$ on $\Omega.$
\end{lemma}
\begin{proof}
One has to show that
\begin{equation}\label{distribucio}
-\int f(z)\,\dbar \varphi(z)\,dA(z) = \int \dbar
f(z)\,\varphi(z)\,dA(z),
\end{equation}
for each $\varphi \in C^\infty_0(\Omega).$ Take $\rho \in
C^\infty_0(\C)$ with support contained in $\{z:|z| \le 1\}$ and
$\int \rho(z) \,dA(z) = 1.$ Set
$\rho_\epsilon(z)=\frac{1}{\epsilon^2} \rho(\frac{z}{\epsilon})$ and
$f_\epsilon = f * \rho_\epsilon.$ Then $\dbar(f*\rho_\epsilon)= f*
\dbar \rho_\epsilon, $ but it is not clear that this coincides with
$\dbar f* \rho_\epsilon.$ Given $\delta> 0$ set $\Omega_\delta =\{z
\in \Omega : \operatorname{dist}(z, \partial \Omega) > \delta\}$ and
take $\delta$ small enough so that the support of $\varphi$ is
contained in $\Omega_\delta .$  Let $\epsilon < \delta.$ Then
\begin{equation}\label{distribucions}
\begin{split}
-\int f_\epsilon(z)\,\dbar \varphi(z)\,dA(z)   & = - \int
(f*\rho_\epsilon)(z)\, \dbar \varphi(z)\,dA(z),
\\*[5pt] & =  \int
\dbar (f*\rho_\epsilon)(z)\,  \varphi(z)\,dA(z)
\\*[5pt] & = \int
 (f* \dbar \rho_\epsilon)(z) \, \varphi(z)\,dA(z).
\end{split}
\end{equation}
We show now that
\begin{equation}\label{quadrat} (f* \dbar
\rho_\epsilon)(z)= (\dbar f * \rho_\epsilon)(z), \quad z \in
\Omega_\delta.
\end{equation}
 Inserting this in \eqref{distribucions} and letting $\epsilon
\rightarrow 0$ yields \eqref{distribucio}. To prove \eqref{quadrat}
let $Q$ be the closed square with sides parallel to the coordinate
axis with center $z$ and side length $2\epsilon.$ Then $Q \subset
\Omega.$ The function $g(w)= f(w) \rho_\epsilon(z-w)$ vanishes on
the boundary of $Q.$ Applying \eqref{Fesq2} to $Q$ and $g$ we get
\eqref{quadrat}.
\end{proof}

\begin{gracies}
The authors are grateful to J.Bruna and M.Melnikov for some useful
conversations on the subject.

This work was partially supported by the grants 2009SGR420
(Generalitat de Catalunya) and  MTM2010-15657 (Ministerio de
Educaci\'{o}n y Ciencia).
\end{gracies}

\vspace*{.55cm}

\begin{tabular}{l}
Juli\`{a} Cuf\'{\i} and Joan Verdera\\
Departament de Matem\`{a}tiques\\
Universitat Aut\`{o}noma de Barcelona\\
08193 Bellaterra, Barcelona, Catalonia\\
{\it E-mail:} {\tt jcufi@mat.uab.cat}\\
{\it E-mail:} {\tt jvm@mat.uab.cat}
\end{tabular}
\end{document}